\documentclass[11pt, a4paper]{amsart}

\RequirePackage[left=20mm,right=20mm,top=30mm,bottom=30mm]{geometry}

\usepackage{amsmath, amssymb, amsthm}
\usepackage{amscd}
\usepackage[all]{xy}

\usepackage{color}

\usepackage{tikz-cd}
\usepackage[square,sort,comma,numbers]{natbib}
 \usepackage{url}
 \usepackage{hyperref}
 \usepackage[shortlabels]{enumitem} 
\usepackage{mathabx} 

\newtheorem{numberingthm}{Theorem}[section] 
\theoremstyle{definition}
\newtheorem{Def}[numberingthm]{Definition}
\newtheorem{example}[numberingthm]{Example}

\theoremstyle{plain}
\newtheorem{Prop}[numberingthm]{Proposition}
\newtheorem{Theorem}[numberingthm]{Theorem}

\newtheorem{Cor}[numberingthm]{Corollary}
\newtheorem{Lemma}[numberingthm]{Lemma}

\newtheorem*{thmintroduction}{Theorem}

\newenvironment{Example}
{\pushQED{\qed}\example}
{\popQED\endexample}

\theoremstyle{remark}

\newtheorem{Remark}[numberingthm]{Remark}

\newcommand{\m}{\!\operatorname{-mod}}
\newcommand{\Ext}{\mathsf{Ext}}
\newcommand{\Tor}{\mathsf{Tor}}
\newcommand{\End}{\mathsf{End}}
\newcommand{\lerp}{{}^\perp}
\newcommand{\add}{\!\mathsf{add}}
\newcommand{\pdim}{\mathsf{pd}}

\newcommand{\Hom}{\mathsf{Hom}}
\newcommand{\injdim}{\mathsf{id}}
\newcommand{\gldim}{\mathsf{gldim}}
\newcommand{\findim}{\mathsf{findim}}
\newcommand{\op}{\mathsf{op}}
\newcommand{\hy}{\!\operatorname{-}\!}

\DeclareMathOperator*{\lsmod}{\!-\mathsf{mod}}

\DeclareMathOperator{\ldom}{\!-\mathsf{domdim}} 
\DeclareMathOperator{\lcodom}{\!-\mathsf{codomdim}}

\begin{document}

\baselineskip=14pt

\title{Relative Auslander--Gorenstein Pairs}

\author[T. Cruz]{Tiago Cruz}
\address[Tiago Cruz]{Max-Planck-Institut f\"ur Mathematik, Vivatsgasse 7, 53111 Bonn, Germany}
\email{tiago.cruz@mathematik.uni-stuttgart.de}
\curraddr{Institut f\"ur Algebra und Zahlentheorie, Universit\"at Stuttgart, Germany }

\author[C. Psaroudakis]{Chrysostomos Psaroudakis}
\address[Chrysostomos Psaroudakis]{Department of Mathematics, Aristotle University of Thessaloniki,
   54124, Thessaloniki, Greece}
\email{chpsaroud@math.auth.gr}


\begin{abstract}
In this paper, we introduce and study relative Auslander--Gorenstein pairs. This consists of a finite-dimensional Iwanaga-Gorenstein algebra together with a self-orthogonal module that provides a further homological feature of the algebra in terms of relative dominant dimension. These pairs will be called relative Auslander pairs whenever the algebra in question has finite global dimension. We characterize relative Auslander pairs by the existence and uniqueness of tilting-cotilting modules having higher values relative dominant and codominant dimension with respect to the self-orthogonal module. The same characterisation remains valid for relative Auslander--Gorenstein pairs if the self-orthogonal module has injective or projective dimension at most one. Our relative approach \mbox{generalises and unifies} the known results from the literature, for instance, the characterization of minimal Auslander--Gorenstein algebras. As an application of our methods, we prove that for any relative Auslander pair pieces of the module category of the endomorphism algebra of the self-orthogonal module can be identified with pieces of the module category of the endomorphism algebra of the unique tilting-cotilting module associated with the relative Auslander pair.  We provide explicit examples of relative Auslander pairs.
\end{abstract}

\maketitle



\section{Introduction and the main result}

One of the most important problems in representation theory is to determine whether a given finite-dimensional algebra has a finite number of indecomposable modules, that is, if it has finite representation type. Auslander's correspondence \cite{zbMATH03517355} connects this problem with homological algebra. Indeed, the endomorphism algebra of the direct sum of all non-isomorphic indecomposable modules over a finite-dimensional algebra of finite representation type has dominant dimension at least two and global dimension at most two. This endomorphism algebra is known as Auslander algebra.

In 70’s and 80’s, tilting theory emerged as a very useful technology in representation theory, and in other areas like algebraic geometry and algebraic topology. Tilting modules are useful objects to create derived equivalences. It was soon realised that they are a powerful tool  in the study of representation-finite type algebras and their indecomposable modules.  
 In \cite{zbMATH06685118}, a characterisation of Auslander algebras was found using tilting theory. More precisely, in \cite{zbMATH06685118}, Auslander algebras were characterised by the existence of tilting modules being generated and cogenerated by faithful projective-injective modules (see also \cite{zbMATH07081686}). Auslander algebras admit higher versions, the so-called $n$-Auslander algebras (see \cite{IyamaCorrespondence}) and these also can be characterised by the existence and uniqueness of tilting modules having large dominant and codominant dimension (see for example \cite{zbMATH07441895} and also \cite{zbMATH06902477}). In particular, the faithful projective-injective is a summand of such a tilting module. Similar results and generalisations of Auslander-type algebras have appeared since then in \mbox{\cite{AT, LiZhang, zbMATH07198564, Grevstad}.} 

The purpose of this paper is to extend these results on uniqueness and existence of tilting modules  which are determined by a module having faithful dimension (in the sense of \cite{BuanSolberg}) greater or equal to the global dimension of the algebra. Indeed, we call such a pair formed by the algebra and the module in the above conditions a relative Auslander pair (see Definition~\ref{defAusGorpair}).

In this paper, we suggest that the concept of relative dominant dimension with respect to a module studied and developed in \cite{Cr2} is the best suited framework to investigate these pairs and the tilting module associated with these covers.  Observe that when the module in question is projective-injective, we recover the classical dominant dimension and when is injective we recover the homological invariant 
used in \cite{LiZhang, AT}. A particular case of the main result of this paper is summarized below and is proved in Theorem~\ref{mainthm}.

\begin{thmintroduction}
Let $A$ be a finite-dimensional algebra and let $Q$ be a finitely generated, self-orthogonal $A$-module, with projective and injective dimensions at most $d$ for some natural number $d$. The following statements are equivalent$\colon$
\begin{enumerate}
\item $(A, Q)$ is a relative $2d$-Auslander pair, that is, $A$ has global dimension at most $2d$ and the  relative dominant dimension of $A$ with respect to $Q$ is greater or equal to $2d$.

\item There exists a unique basic $d$-tilting-cotilting module $T$ satisfying the conditions
\begin{enumerate}[(i)]
\item $Q\ldom_AT\geq d$, 

\item $Q\lcodom_AT\geq d$, and

\item  The  right $\Ext$-orthogonal modules of $T$ are exactly the modules having a finite resolution by modules in the additive closure of $T$.
\end{enumerate}
\end{enumerate}
\end{thmintroduction}

According to \cite{zbMATH02105773} and \cite{CE}, if $A$ is a quasi-hereditary algebra over an algebraically closed field with a simple preserving duality, then the tilting module associated with any relative Auslander pair $(A, Q)$ where $Q$ is a summand of the characteristic tilting module is exactly the characteristic tilting module. In Theorem \ref{mainthm}, we cover also the relative odd-Auslander pairs, but in such a case we no longer have a symmetry in (2)(i)-(ii). A feature of all relative Auslander pairs is that the tilting-cotilting $T$ associated with a relative Auslander pair always satisfies $Q\ldom_A T+Q\lcodom_A T\geq \gldim A$.

Our strategy for the main result makes use of the fact that tilting modules over finite-dimensional algebras having finite global dimension are also cotilting modules and the cotorsion pairs induced by being tilting and cotilting are exactly the same (see Theorem \ref{thm4dot1}).

If $Q$ has projective dimension at most one, then relative $2d$-Auslander pairs can also be characterised in terms of the $r$-tilting modules with $r\in \{1, \ldots,  d\}$ (see Corollary \ref{cor3dot4}).  Along the way, we introduce the concept of relative $n$-Auslander--Gorenstein pairs which generalise the concept of minimal Auslander--Gorenstein algebra. If $Q$ has projective dimension one, the relative $n$-Auslander--Gorenstein pairs are also characterised by the existence and uniqueness of a certain tilting-cotilting module in terms of relative dominant dimension (see Theorem \ref{thm3dot3}.)

 At the end of the paper, we exhibit a connection between relative Auslander pairs and cover theory. In fact, we show in Theorem \ref{thm7dot4} that any relative $2d$-Auslander pair $(A, Q)$ induces an $d$-$\mathcal{A}$ cover of $\End_A(Q)^{op}$ for a certain resolving subcategory $\mathcal{A}$ of a module category of a finite-dimensional algebra which is derived equivalent to $A$.  We conclude the paper by illustrating examples of relative Auslander pairs.
 Blocks of Schur algebras of finite type are Auslander algebras, and so in particular they form relative Auslander pairs with its faithful projective-injective module. The recent results of \cite{CE} show that the Schur algebras $S(2, d)$ together with the tensor power form relative Auslander pairs.

The paper is organised as follows: in Section~\ref{Preliminaries}, we discuss properties of modules $X$ with large relative (resp. codominant) dominant dimensions with respect to a module without self-extensions. In Section  \ref{sec3}, we establish that the faithful dimension of a self-orthogonal module $Q$ of projective dimension at most one characterises the existence of a certain $d$-tilting module having relative codominant dimension at least $d$ with respect to $Q$.
	In Section~\ref{sec4}, we introduce the concepts of relative Auslander--Gorenstein pairs and how they are characterised under certain conditions by the existence of a certain tilting-cotilting module.
	In Section~\ref{sec5}, we make precise how tilting modules over a finite-dimensional algebra $A$ can detect whether $A$ has finite global dimension.
	In Section~\ref{relativeGorAuspairs}, we establish that all relative Auslander pairs $(A, Q)$ are characterised by the existence of certain tilting-cotilting modules having in particular $Q$ as direct summand. In Section~\ref{sec7}, we prove that the endomorphism algebra of any module $Q$ that fits into an Auslander pair $(A, Q)$ can be resolved by an algebra derived invariant to $A$. Further, in Section~\ref{Examples}, we present examples of Auslander pairs.

\section{Preliminaries}\label{Preliminaries}

Throughout the paper, we denote by $A$ a finite-dimensional algebra over a field $k$ and $d$ a natural number. By $A\lsmod$ we denote the category of finitely generated left $A$-modules. Let  $Q$ and $X$ be in $A\lsmod$. We write $\lerp Q=\{Y\in \lsmod \ | \ \Ext_A^i(Y, Q)=0, \, \forall i>0 \}$ and ${D=\Hom_{k}(-,k)\colon A\m \to A^{op}\m}$ for the standard duality.  By $\add_A X$ (or just $\add X$) we mean the additive closure of $X$. Given $M\in A\m$, we will denote by $\pdim_{A} M$ (resp. $\injdim_A M$) the projective (resp. injective) dimension of $M$ over $A$. By $\gldim A$ we mean the global dimension of $A$. An algebra $A$ is called {\bf Iwanaga-Gorenstein} if $\injdim_A A$ and $\injdim A_A$ are finite. We start by recalling the concept of relative (co)dominant dimension presented in \citep{Cr2}.

\begin{Def}
Let $A$ be a finite-dimensional algebra and $Q, X$ in $A\lsmod$.
\begin{enumerate}
\item[(i)] We say that the {\bf relative dominant dimension} of $X$ with respect to $Q$ is greater or equal to $d$, denoted by $Q\ldom_AX\geq d$, if there is an exact sequence $0\to X\to Q_1\to Q_2\to \cdots \to Q_d$, with $Q_i\in \add{Q}$, such that the next sequence is exact
\[
 \ \ \ \ \xymatrix{
\Hom_A(Q_d, Q) \ar[r]^{}  & \cdots \ar[r]^{} & \Hom_A(Q_2, Q) \ar[r]^{} & \Hom_A(Q_1, Q) \ar[r]^{} & \Hom_A(X, Q) \ar[r]^{} &  0.
 } 
\]
Otherwise, we say that the {\bf relative dominant dimension} of $X$ with respect to $Q$ is zero.  The {\bf relative dominant dimension} of $X$ with respect to $Q$ is infinite if it is greater or equal to any natural number. 

\item[(ii)] We say that the {\bf relative codominant dimension} of $X$ with respect to $Q$ is greater or equal to $d$, denoted by $Q\lcodom_AX\geq d$, if there is an exact sequence $Q_d\to \cdots\to Q_2\to Q_1\to X \to 0$, with $Q_i\in \add{Q}$, such that the next sequence is exact
\[
\ \ \ \ \xymatrix{
\Hom_A(Q, Q_d) \ar[r]^{}  & \cdots \ar[r]^{} & \Hom_A(Q, Q_2) \ar[r]^{} & \Hom_A(Q, Q_1) \ar[r]^{} & \Hom_A(Q, X) \ar[r]^{} &  0.
 } 
\]
Otherwise, we say that the {\bf relative codominant dimension} of $X$ with respect to $Q$ is zero.
\end{enumerate}
\end{Def}

We recall now some facts from \cite{Cr2} about relative dominant dimension that are going to be used often in what follows. If $Q$ is a projective-injective module, then the relative (resp. co)dominant dimension of $M$ with respect to $Q$ is exactly the (resp. co)dominant dimension of $M$.

For modules belonging to $\lerp Q$, the computation of relative dominant dimension with respect to $Q$ can be simplified. In fact, given $M, Q\in \lerp Q$, $Q\ldom_A M\geq n$ if and only if there exists an exact sequence $0\rightarrow M\rightarrow Q_1\rightarrow \cdots \rightarrow Q_n$ with all $Q_i$ in $\add Q$ and the cokernel of $Q_{n-1}\rightarrow Q_n$ in $\lerp Q$ (see \citep[Proposition 3.1.11]{Cr2}). 

Actually, the relative dominant dimension of $M$ and the cokernel of $Q_{n-1}\rightarrow Q_n$ with respect to $Q$ are related as follows:

\begin{Cor}
\label{cor2dot2}
Assume that $Q\in \lerp Q$ and let $M\in A\m$. Suppose that there exists an exact sequence $0\rightarrow M\rightarrow Q_1\rightarrow \cdots\rightarrow Q_t\rightarrow X\rightarrow 0$ with $Q_i\in \add{Q}$ for some $t>0$. If $\Ext_A^i(X, Q)=0$ for $1\leq i\leq t$, then $Q\ldom_A M=t+Q\ldom_A X$.
\end{Cor}\begin{proof}
See \citep[Corollary 3.1.12]{Cr2}.
\end{proof}

The following lemma gives a sufficient condition to obtain a lower bound on the relative dominant dimension if we only know that the cokernel of $Q_{n-1}\rightarrow Q_n$ is cogenerated by $Q$.

\begin{Lemma}
\label{lemmaextendingdom}
Let $Q, \ M\in A\m$.
	Assume that $Q\ldom_A M\geq n\geq 1$ where the $A$-exact sequence
	\begin{align}
		0\rightarrow M\rightarrow Q_1\rightarrow \cdots \rightarrow Q_{n-1}, 
	\end{align}with $Q_i\in \add Q$, which remains exact under $\Hom_A(-, Q)$, can be continued to an
	$A$-exact sequence
	\begin{align}
		0\rightarrow M\rightarrow Q_1\rightarrow \cdots \rightarrow Q_{n-1}\rightarrow Y
	\end{align}with $Y\in \add Q$. Then, $Q\ldom_A M\geq n+1$.
\end{Lemma}
\begin{proof}
	See \citep[Lemma 3.1.9]{Cr2}.
\end{proof}

Observe that for every $Q, M\in A\m$, $Q\ldom_A M= DQ\lcodom_{A^{op}} DM$ (see for example \citep[Subsection 2.1]{Cr2}). As a consequence, all the previous statements can be dualised for relative codominant dimension.

If we fix $M$ to be the regular module $A$, the previous equality can be expanded to the equality $DQ\ldom_A A = Q\lcodom_{A^{op}} DA= Q\ldom_A A$ (see \citep[Corollary 3.1.5]{Cr2}). Here, the value $Q\ldom_A A$ has appeared in the literature before as the faithful dimension of $Q$ in the sense of Buan and Solberg \cite{BuanSolberg} (see \citep[Subsection 3.1]{Cr2}).

We shall now explore that modules having large relative dominant (and codominant) dimension with respect to $Q$ are $\Ext$-orthogonal to $Q$.

\begin{Lemma}
\label{lemma1dot1}
Let $Q\in A\lsmod$ and assume that there is an exact sequence
\[
0\rightarrow X_d'\rightarrow Q_d'\rightarrow \cdots \rightarrow Q_1'\rightarrow X_0'\rightarrow 0,
\]
where $Q_i'\in \add Q$. If $X_0\in \lerp Q$, then $\Ext_A^i(X_0, X_0')\simeq \Ext_A^{i+d}(X_0, X_d')$ for all $i>0$.
\end{Lemma}
\begin{proof}
	Denote by $X_i'$ the kernel of $Q_i'\rightarrow Q_{i-1}'$ and $Q_0':=X_0'$. Applying $\Hom_A(X_0, -)$ to the exact sequences 
\[
0\rightarrow X_{i+1}'\rightarrow Q_{i+1}'\rightarrow X_i'\rightarrow 0 
\]
we obtain that $\Ext_A^l(X_0, X_i')\simeq \Ext_A^{l+1}(X_0, X_{i+1}')$ for $l>0$ because $X_0$ is in $\lerp Q$. Hence, for any $l>0$, we get that $\Ext_A^l(X_0, X_0')\simeq \Ext_A^{l+1}(X_0, X_1')\simeq \Ext_A^{l+d}(X_0, X_d')$.
\end{proof}

Objects living  in $\lerp Q$ arise naturally as objects having relative dominant dimension with respect to $Q$ greater or equal to the injective dimension of $Q$.

\begin{Cor}\label{Cor1dot3}
	Let $Q\in A\m$ satisfying $\Ext_A^{i>0}(Q,Q)=0$. The following assertions hold.
	\begin{enumerate}
				\item[(i)] If $\injdim_A Q\leq d$ and $Q\ldom_A X\geq d$, then $X\in \lerp Q$.
			\item[(ii)] If $\pdim_A Q\leq d$ and $Q\lcodom_A X\geq d$, then $X\in Q^{\perp}$.
			\item[(iii)]  Let $X\in \lerp Q$ and $Y\in A\m$ satisfying $\pdim_{A} X\leq d\leq Q\lcodom_A Y$. Then, $\Ext_A^{i>0}(X, Y)=0$.
		\item[(iv)] Assume that $\injdim_A Q\leq d$. Let $X, Y\in A\m$ satisfying 
\[
\sup\big\{\pdim_A X, \pdim_A Y \big\}\leq d\leq \inf\big\{Q\lcodom_A X, Q\lcodom_A Y, Q\ldom_A X, Q\ldom_A Y \big\}.
\]
Then, $\Ext_A^i(X, Y)=0$ for every $i>0$.
	\end{enumerate}
\end{Cor}
\begin{proof}
(i)	Since $Q\ldom_A X\geq d$, there exists an $A$-exact sequence $0\rightarrow X\rightarrow Q_1\rightarrow \cdots\rightarrow Q_d$ which remains exact under $\Hom_A(-, Q)$ and $Q_i\in \add Q$. Denote by $X_i$ the image of $Q_i \rightarrow Q_{i+1}$ and by $X_d$ the cokernel of $Q_{d-1}\rightarrow Q_d$. By applying $\Hom_A(-, Q)$ to the above exact sequence and since $\Ext_A^{i>0}(Q,Q)=0$, it follows that $\Ext_A^i(X, Q)\simeq \Ext_A^{i+1}(X_1, Q)\simeq \Ext_A^{i+d}(X_d, Q)$ for all $i>0$. The latter extension group is zero for all $i>0$ since $\injdim_A Q\leq d$. We infer that $X$ lies in $\lerp Q$.
	
(ii) It follows by (i) together with the observation that $Q\ldom_A X = DQ\ldom_{A^{op}}DX$.

	(iii)  Since $Q\lcodom_A Y\geq d$ there exists an exact sequence $0\rightarrow X_d'\rightarrow Q_d'\rightarrow \cdots \rightarrow Q_1'\rightarrow Y\rightarrow 0$
	where $Q_i'\in \add Q$ and remains exact after applying the functor $\Hom_A(Q,-)$. Then by Lemma~\ref{lemma1dot1} we have the isomorphism $\Ext^l_A(X, Y)\cong \Ext_A^{l+d}(X, X'_d)$, $\forall l>0$, but the latter extension group is zero since $\pdim_A X\leq d$. We conclude that $\Ext_A^i(X, Y)=0$ for all $i>0$.

(iv) By (i), the module $X$ lies in $\lerp Q$. The claim then follows by (iii).
\end{proof}

A module $T\in A\m$ is called {\bf tilting} if it has finite projective dimension over $A$, $T\in \lerp T$ and there exists an exact sequence $0\rightarrow A\rightarrow T_0\rightarrow T_1\rightarrow \cdots\rightarrow T_s\rightarrow 0$ with $T_i\in \add_AT$ for some $s$. $T$ is said to be {\bf $d$-tilting} if it is tilting and $T$ has projective dimension at most $d$. Dually, $T$ is called {\bf cotilting} (resp. {\bf $d$-cotilting}) over $A$ if $DT$ is tilting (resp. $d$-tilting) over $A^{\op}$.

The following result together with the following remark generalise \citep[Lemma~3.11]{AT}.

\begin{Lemma}
\label{lemma1dot4}
		Let $Q\in A\m$ satisfying $\Ext_A^{i>0}(Q,Q)=0$ and $\pdim_A Q\leq d$, $\injdim_A Q\leq n-d$. Then, up to isomorphism, there is at most one basic $d$-tilting module $T$ satisfying
		\begin{enumerate}
			\item[(i)] $Q\ldom_A T\geq n-d$;
			\item[(ii)] $Q\lcodom_A T\geq d.$
		\end{enumerate}In particular, if it exists, then $Q\in \add_A T$.
\end{Lemma}
\begin{proof}
	Define $T_1$ to be the direct sum of all non-isomorphic indecomposable modules $X$ in $^{\perp}Q$ satisfying 
\begin{align*}
		\pdim_A X\leq d\leq Q\lcodom_A X
\end{align*}
In particular, $\pdim_A T_1\leq d$ and $Q\in \add_A T_1$. Since $\injdim_AQ\leq n-d$ and $Q\ldom_A T\geq n-d$, Corollary~\ref{Cor1dot3} (i) gives that $T\in \lerp Q$. Hence, $T\in \add_A T_1$. By Corollary ~\ref{Cor1dot3} (iii),  $\Ext_A^{i>0}(T_1, T_1)=0$.
 Hence, $T_1$ is a tilting $A$-module. It follows that the number of non-isomorphic indecomposable summands of $T$ and $T_1$ must coincide, and therefore, $T\simeq T_1$.
\end{proof}

\begin{Remark}\label{remconditiondualadd}
We point out that Lemma \ref{lemma1dot4} has several variations:
\begin{enumerate}
\item[(i)] According to Lemma~\ref{lemma1dot1}, the condition $Q\ldom_A T\geq n-d$ in Lemma~\ref{lemma1dot4} can be replaced by imposing that $Q$ also satisfies the following ${\{X\in A\m \ | \ Q\lcodom_A X\geq d\}\subset \lerp Q}$  or just that $T\in \lerp Q$ fixing $n=2d$. Indeed, in either case, $T_1$ can be defined as the direct sum of all non-isomorphic indecomposable modules $X$ belonging to $\lerp Q$ satisfying ${\pdim_A X\leq d\leq Q\lcodom_A X}$.

\item[(ii)] Since $T$ is a tilting $A$-module if and only if $DT$ is a cotilting $A^{op}$-module, Lemma  \ref{lemma1dot4} also holds for basic $(n-d)$-cotilting modules.

\item[(iii)] Let $Q\in A\m$ satisfying $\Ext_A^{i>0}(Q,Q)=0$ and $\pdim_{A} Q\leq d$ and $\injdim_AQ \leq n-d$. Assume further that there exists a tilting module $_AT$ satisfying
\[
{Q\ldom}_A T\geq n-d \quad \text{and} \quad {Q\lcodom}_A T \geq d.
\]
Then $Q$ lies in $\add_A T$. Indeed, by Corollary~\ref{Cor1dot3} we have $T\in {^\perp}Q$ and $T\in Q^{\perp}$. Hence $\Ext_A^{i>0}(Q\oplus T, Q\oplus T)=0$ and therefore the module $T\oplus Q$ is tilting. So $\add(T\oplus Q) = \add T$ which means that $Q\in \add_AT$.
\end{enumerate}
\end{Remark}

We finish this preliminary section by recalling the notion of relative projective dimension needed in Section~\ref{relativeGorAuspairs}. Let $\mathcal{X}$ be a full subcategory in $A\m$. The {\bf $\mathcal{X}$-projective dimension} of $N\in A\m$ is defined as the following value:
\[
\mathsf{inf}\big\{n\in \mathbb{N}\cup \{0\}  \ | \ \Ext_A^{i>n}(N, M)=0, \ \forall M\in \mathcal{X}   \big\}
\]
We write $\pdim_{\mathcal{X}}N$ for the above relative projective dimension of $N$ with respect to $\mathcal{X}$.

\section{Existence of tilting modules with higher relative dominant dimension}\label{sec3}

In this section, we will assume that $\pdim_A Q\leq 1$. In this case, Adachi and Tsukamoto proved in \citep[Proposition 3.5(2)]{AT} that the existence of tilting modules satisfying $Q\lcodom_A T\geq d$ completely determines the exact coresolution of $A$ by $\add_A T$.

\begin{Prop}\label{Prop2dot1}
	Let $Q\in A\m$ satisfying $\Ext_A^{i>0}(Q,Q)=0$, $\pdim_A Q\leq 1$ and $\injdim_A Q\leq d$. For a $d$-tilting module $_AT$ consider the following conditions$\colon$
\begin{enumerate}
	\item[(i)] $\{X\in A\m \ | \  Q\lcodom_A X\geq d\}\subset \lerp Q$;
	\item[(ii)] $T\in \lerp Q$;
	\item[(iii)] $Q\ldom_A T\geq d$;
\end{enumerate}	
If one of the above holds and $Q\lcodom_A T\geq d$, then $Q\in \add{_AT}$. In this case, there exists an exact sequence
	\begin{align}
		0\rightarrow A\rightarrow Q_1\rightarrow Q_2\rightarrow \cdots \rightarrow Q_d\rightarrow T_d\rightarrow0, \label{eq1}
	\end{align}
with $Q_i\in \add_A Q$, $T_d\in \add_A T$ and which remains exact under $\Hom_A(-, Q)$.	
\end{Prop}
\begin{proof} If $Q\lcodom_A T\geq d$ together with one of the conditions (i), (ii) and (iii), then $Q\in \add_A T$ by Lemma \ref{lemma1dot4} and Remark \ref{remconditiondualadd}.
	For the existence of the exact sequence, see Proposition 3.5(2) of \cite{AT}. Since $Q\in \add_AT$ such exact sequence (\ref{eq1}) remains exact under $\Hom_A(-, Q)$.
\end{proof}

The following provides a higher version of \citep[Theorem 3.10]{AT}.

\begin{Theorem}
\label{thm2dot2}
Let $Q\in A\m$ satisfying $\Ext_A^{i>0}(Q,Q)=0$, $\pdim_A Q\leq 1$ and $\injdim_A Q\leq d$. Then, $Q\ldom_A A\geq 2d$ if and only if there exists a unique basic $d$-tilting module $T$ satisfying 
\begin{equation}
\label{infQdomQcodom}
\inf\big\{Q\ldom_A T, Q\lcodom_A T \big\}\geq d.
\end{equation}
\end{Theorem}
\begin{proof}
	Assume that there exists a unique basic $d$-tilting module $T$ satisfying the relation $(\ref{infQdomQcodom})$. By Proposition~\ref{Prop2dot1} and  Corollary~\ref{cor2dot2}, we have
\[
Q\ldom_A A=d+Q\ldom_A T\geq 2d.
\]
Conversely, assume that $Q\ldom_A A\geq 2d$. By definition, there exists an $A$-exact sequence
\begin{align}
	0\rightarrow A\rightarrow Q_1\rightarrow \cdots\rightarrow Q_{2d}
\end{align}which remains exact under $\Hom_A(-, Q)$ and $Q_i\in \add_AQ$. Denote by $X$ the image of $Q_{d}\rightarrow Q_{d+1}$ and by $Y$ the cokernel of $Q_{d}\rightarrow Q_{d+1}$. Fix $T=Q\oplus X$. By Corollary~\ref{cor2dot2}, $Q\ldom_A X=Q\ldom_A A -d \geq d$. Hence, $Q\ldom_AT\geq d$. Denote by $C$ the cokernel of $A\rightarrow Q_1$. Thus, $\pdim_A C\leq 1$ since $\pdim_A Q\leq 1$. So, the exact sequence $0\rightarrow C\rightarrow Q_2\rightarrow \cdots\rightarrow Q_d\rightarrow X\rightarrow 0$ gives that $\pdim_A X\leq d$, and therefore $\pdim_A T\leq d$. Since $\pdim_A Q\leq 1$, applying $\Hom_A(Q, -)$ to $0\rightarrow A\rightarrow Q_1\rightarrow C\rightarrow 0$ yields that $C$ belongs to $Q^\perp$. Analogously, we obtain that $Y$ lies in $Q^\perp$.
By the dual of \citep[Proposition~3.1.11]{Cr2} the sequence $0\rightarrow C\rightarrow Q_2\rightarrow \cdots\rightarrow Q_d\rightarrow Q_{d+1}\rightarrow Y\rightarrow 0$ gives that $Q\lcodom_A Y\geq d$. By the dual of Lemma~\ref{lemmaextendingdom}, we obtain that $Q\lcodom_A Y\geq d+1$. By the dual of \citep[Lemma 3.1.7]{Cr2} on the exact sequence $0\rightarrow X\rightarrow Q_{d+1}\rightarrow Y\rightarrow 0$, we deduce that $Q\lcodom_A X\geq d$. Finally, by Lemma~\ref{lemma1dot4}, the uniqueness of $T$ follows.
\end{proof}

Note that by fixing $Q$ to be projective-injective and $d=1$ we recover \citep[Theorem 3.3.4]{zbMATH07081686}.

\begin{Remark}
 Observe that if $Q$ is injective, then $Q$ satisfies the condition (i) of Proposition \ref{Prop2dot1}. Hence, Corollary \ref{cor2dot2} implies that if there exists a $d$-tilting module with $1\leq d\leq \min\{\injdim_A A, Q\lcodom_A T \}$, then $Q\ldom_A A=d+Q\ldom_A T$. This equality also holds if $d=0$ because in such a case $\add T=\add A$. 
\end{Remark}

 Therefore, copying the proof of Theorem \ref{thm2dot2} (replacing the use of Lemma~\ref{lemma1dot4} by Remark \ref{remconditiondualadd}) in case that $1\leq d\leq \min\{Q\ldom_A A, \injdim_A A\}$ and $Q$ is injective, we obtain a sharper version of \citep[Theorem 3.10]{AT} which is one of the main results of \cite{AT}.

\begin{Cor}
\label{cor3dot4}
Let $Q$ be an injective $A$-module with projective dimension at most one. 
For any natural number $n$ and $1\leq d\leq \min\{\injdim_AA, n\}$, there exists a unique basic $d$-tilting module $T$ satisfying $Q\lcodom_A T\geq d$ and $Q\ldom_A T\geq n-d$ if and only if $Q$ has faithful dimension at least $n$.
\end{Cor}
 
The same argument holds if instead of requiring one of the enumerated assumptions of Proposition \ref{Prop2dot1} we require directly that $Q\in \add_AT$. In such a case, the condition $\injdim_AQ\leq d$ does not influence the existence of the exact sequence and so it can be dropped.

\begin{Cor}
Let $Q\in A\m$ satisfying $\Ext_A^{i>0}(Q,Q)=0$ and $\pdim_A Q\leq 1$. For any natural number $n$ and $1\leq d\leq \min\{\injdim_AA, n\}$, there exists a unique basic $d$-tilting module $T$ satisfying $Q\lcodom_A T\geq d$, $Q\ldom_A T\geq n-d$ and $Q\in \add_A T$ if and only if $Q$ has faithful dimension at least $n$.
\end{Cor}

By fixing $Q$ to be projective-injective the above result generalises \citep[Theorem 3.2]{zbMATH07441895}.

\section{Relative Gorenstein--Auslander Pairs}\label{sec4}

In this section, we will propose a generalisation of the concepts of minimal Auslander--Gorenstein algebras in the sense of \cite{IS} and almost Auslander algebras in the sense of \cite{AT}. Further, our goal in this section is to show that this new concept, under suitable conditions, can be characterised in terms of the existence of a certain tilting-cotilting module (see Theorem \ref{thm3dot4}).

\begin{Lemma}\label{lemma4dot1}
Let $Q\in A\m$ satisfying $\Ext_A^{i>0}(Q,Q)=0$, $\pdim_A Q\leq 1$ and $\injdim_A Q\leq d$.
	Assume that $T$ is a $d$-tilting-cotilting module satisfying $\inf\{Q\ldom_AT, Q\lcodom_A T \}\geq d$. Then, $\injdim_A A\leq 2d$ and $\pdim_A DA\leq 2d$, that is, $A$ is $2d$-Iwanaga-Gorenstein.
\end{Lemma}
\begin{proof}
By Theorem \ref{thm2dot2} and Proposition \ref{Prop2dot1}, there exists an exact sequence
	\begin{align}
		0\rightarrow A\rightarrow Q_1\rightarrow Q_2\rightarrow \cdots\rightarrow Q_d\rightarrow X\rightarrow 0,
\label{eq7}	\end{align}where $Q_i\in \add_A Q$ and $\add_A T=\add_A(Q\oplus X)$. In particular, $\injdim_A X\leq d$ since $T$ is $d$-cotilting. Recall that if there exists an exact sequence $0\rightarrow N_1\rightarrow N\rightarrow N_2\rightarrow 0$ with $\injdim_A N\leq d\leq s$ and $\injdim_A N_2\leq s$ for some $s\geq 0$, then $\injdim_A N_1\leq s+1$. Therefore, using an inductive argument, (\ref{eq7}) gives that $\injdim_A A\leq 2d$.

Since $Q\ldom_A A=Q\lcodom_A DA\geq 2d$ (see \citep[Corollary 3.1.5]{Cr2}) we obtain an exact sequence
\begin{align}
	0\rightarrow Y\rightarrow Q_{d-1}\rightarrow \cdots\rightarrow Q_0\rightarrow DA\rightarrow0 \label{eq8}
\end{align}which remains exact under $\Hom_A(Q, -)$. Since $\injdim_A T\leq d$ we have $\Ext_A^{i\geq d+1}(DA, T)=0$.  Moreover, using the fact that $T$ is a $d$-cotilting module, there exists an exact sequence $0\rightarrow T_d\rightarrow \cdots \rightarrow T_0\rightarrow DA\rightarrow 0$, where $T_i\in \add_A T$. Since $T$ has projective dimension at most $d$, it follows that $DA$ has finite projective dimension. Suppose, for now, that $Y\in T^\perp$. Recall that every element in $T^\perp$ admits a resolution (possibly infinite) by $\add T$, see \cite[Theorem~5.4 (b)]{AuslanderReiten}. So, the exact sequence (\ref{eq8}) can be extended to an exact sequence of the form \begin{align}
\cdots \rightarrow T_r\rightarrow \cdots\rightarrow T_d\rightarrow Q_{d-1}\rightarrow \cdots\rightarrow Q_0\rightarrow DA\rightarrow 0,
\end{align}where $T_i\in \add_A T$. This resolution must be finite since $DA$ has finite projective dimension. Now, using the fact that $\Ext_A^{i\geq d+1}(DA, T)=0$ this exact sequence must terminate at $T_d$, otherwise it would split. It follows that $Y\simeq T_d\in \add_A T$, and so $\pdim_A Y\leq d$. It follows by (\ref{eq8}) that $\pdim_A DA\leq 2d$.

It remains to show that $Y\in T^\perp$. By (\ref{eq8}), $Q\lcodom_A Y\geq d$ and so there exists an exact sequence $0\rightarrow Z\rightarrow Q_{2d}\rightarrow \cdots\rightarrow Q_d\rightarrow Y\rightarrow 0$, where $Q_i\in \add_A Q\subset \add_A T$. Hence, by applying $\Hom_A(T, -)$ and using that $\Ext_A^{j}(T, T)=0$ for all $j>0$, we obtain $\Ext_A^{i>0}(T, Y)\simeq \Ext_A^{i+d}(T, Z)$ which is zero because $\pdim_AT\leq d$. This concludes the proof.
\end{proof}

Under the same conditions on the module $Q$ and $A$ being an Iwanaga-Gorenstein algebra, we show below the existence of a unique tilting-cotilting module satisfying further nice properties.

\begin{Lemma}\label{lemma4dot2}
	Let $Q\in A\m$ satisfying $\Ext_A^{i>0}(Q,Q)=0$, $\pdim_A Q\leq 1$ and $\injdim_A Q\leq d$. Assume that $A$ is an Iwanaga-Gorenstein algebra satisfying the following: 
	\begin{align}
		\injdim_A A\leq 2d\leq Q\ldom_A A.
	\end{align}Then, there exists a unique basic $d$-tilting-cotilting module $T$ satisfying 
\[
\inf\big\{Q\ldom_A T, Q\lcodom_A T\big\}\geq d \text{ and } Q\in \add_A T.
\]
\end{Lemma}
\begin{proof}
	By Theorem \ref{thm2dot2}, there exists a unique basic $d$-tilting module $T$ satisfying 
\[
\inf\{Q\ldom_A T, Q\lcodom_A T\}\geq d
\]
and $Q\in \add_A T$. It remains to prove that $T$ is also a $d$-cotilting module. Since $Q\lcodom_A DA\geq 2d$ there exists an exact sequence
\begin{align}
	0\rightarrow Y'\rightarrow Q^{2d}\rightarrow \cdots\rightarrow Q^{d+1}\rightarrow Q^d\rightarrow \cdots \rightarrow Q^1\rightarrow DA\rightarrow 0 \label{eq13}
\end{align}
which remains exact under $\Hom_A(Q, -)$, where $Q^i\in \add_AQ$. Denote by $Y$ the image of $Q^{d+1}\rightarrow Q^d$. So by applying $\Hom_A(T, -)$ to the short exact sequences induced by the previous exact sequence, we obtain that \begin{align}
\Ext_A^i(T, Y)\simeq \Ext_A^{i+1}(T, \ker(Q^{d+1}\rightarrow Y))\simeq \Ext_A^{i+d}(T, Y')=0
\end{align}for all $i>0$. Hence, $Y\in T^\perp$. This implies that we can extend the exact sequence $Q^d\rightarrow \cdots \rightarrow Q^1\rightarrow DA\rightarrow 0$ to an $\add_A T$-resolution. Such resolution must be finite since $\pdim_A DA\leq 2d$. Since $\Ext_A^{i>0}(T, T)=0$, it remains to show that $\injdim_A T\leq d$. To see this, recall that there exists an exact sequence
\begin{align}
	0\rightarrow A\rightarrow Q_1\rightarrow \cdots \rightarrow Q_d \rightarrow X\rightarrow 0 \label{eq12}
\end{align}
which remains exact under $\Hom_A(-, Q)$, where $Q_i\in \add_A Q$ and $\add_A T=\add_A (Q\oplus X)$. Since $\injdim_A Q\leq d$ it follows that the cokernel of $A\rightarrow Q_1$ has injective dimension at most $2d-1$. By an inductive argument on the exact sequence (\ref{eq12}), we conclude that  $X$ has injective dimension at most $d$.
\end{proof}

Combining Theorem \ref{thm2dot2} with Lemmas \ref{lemma4dot1} and \ref{lemma4dot2} we obtain the following main result of this section. 

\begin{Theorem}\label{thm3dot3}
	Let $Q\in A\m$ satisfying $\Ext_A^{i>0}(Q,Q)=0$, $\pdim_A Q\leq 1$ and $\injdim_A Q\leq d$. Then, there exists a unique basic $d$-tilting-cotilting module $T$ with \begin{align}
		\inf\big\{Q\ldom_A T, Q\lcodom_A T\big\}\geq d 
	\end{align} if and only if $A$ is an Iwanaga-Gorenstein algebra satisfying the following: 
\begin{align}
\injdim_A A\leq 2d\leq Q\ldom_A A.
\end{align} In particular, $Q\in \add_A T$.
\end{Theorem}

Observe that we can also state this result for modules $Q$ that have injective dimension at most one and projective dimension at most $d$.

\begin{Theorem}\label{thm3dot4}
	Let $Q\in A\m$ satisfying $\Ext_A^{i>0}(Q,Q)=0$, $\pdim_A Q\leq d$ and $\injdim_A Q\leq 1$. Then, there exists a unique basic $d$-tilting-cotilting module $T$ with \begin{align}
		\inf\{Q\ldom_A T, Q\lcodom_A T\}\geq d 
	\end{align} if and only if $A$ is an Iwanaga-Gorenstein algebra satisfying the following: 
	\begin{align}
		\injdim_A A\leq 2d\leq Q\ldom_A A.
	\end{align} In particular, $Q\in \add_A T$.
\end{Theorem}
\begin{proof}
	By assumption, $DQ\in A^{op}\m$ satisfies  $\Ext_{A^{op}}^{i>0}(DQ,DQ)=0$, $\pdim_A DQ\leq 1$ and $\injdim_A Q\leq d$. Observe that $T$ is a $d$-tilting-cotilting module over $A$ with 
	\begin{align}
		\inf \big\{Q\ldom_A T, Q\lcodom_A T \big\}\geq d 
	\end{align} if and only if  $DT$ is a $d$-tilting-cotilting module over $A^{op}$ with 	\begin{align}
	\inf\big\{DQ\ldom_{A^{op}} DT, DQ\lcodom_{A^{op}} DT\big\}\geq d.
\end{align} 
Therefore, the result follows from  \citep[Corollary 3.1.5]{Cr2} and Theorem \ref{thm3dot3}.
\end{proof}

Based on the above results we define the following notion.

\begin{Def}\label{defAusGorpair}
Let $A$ be a finite-dimensional algebra and $Q\in A\m$ such that $Q\in Q^{\perp}$. The pair $(A, Q)$ is called a {\bf relative $n$-Auslander--Gorenstein pair} if the algebra $A$ is a $n$-Iwanaga-Gorenstein algebra satisfying the following
\[ 
\injdim_A A\leq n\leq Q\ldom_A A
\]

The pair $(A, Q)$ is called a {\bf relative $n$-Auslander pair} if it is a relative $n$-Auslander--Gorenstein pair and $A$ has finite global dimension.
\end{Def}

In particular, if $(A, Q)$ is a relative $n$-Auslander pair, then $\gldim A\leq n$ (see for example \citep[VI, Lemma 5.5]{zbMATH00707210}). 

\begin{Remark}
	If $A$ is an $n$-minimal Auslander--Gorenstein algebra in the sense of \cite{IS}, then the pair $(A, P)$ is a relative $(n+1)$-Auslander--Gorenstein pair, where $P$ is the minimal projective-injective left $A$-module.

	In other words, the concept of relative Auslander--Gorenstein pair that we propose is obtained from replacing the dominant dimension of the algebra in the definition of minimal Auslander--Gorenstein algebras with the faithful dimension in the sense of \cite{BuanSolberg} of a given module being $\Ext$-orthogonal to itself. 
\end{Remark}

\begin{Remark}
	If $A$ is an almost $n$-Auslander--Gorenstein  algebra in the sense of \cite{AT}, then the pair $(A, I)$ is a relative $(n+1)$-Auslander--Gorenstein pair, where $I$ is the direct sum of all indecomposable injective $A$-modules with projective dimension at most one.
\end{Remark}

\begin{Remark}Let $A$ be a finite-dimensional algebra and fix $Q\in A\m$  so that $\pdim_A Q\leq d$, $\injdim_A Q\leq d$ and $Q\in Q^\perp$.

If $A$ has a simple preserving duality $(-)^\natural\colon A\m\rightarrow A\m$, then \begin{align}
	Q\ldom_A A=Q^\natural\lcodom_A DA=DQ^\natural\ldom_{A^{op}} A.
\end{align}Thus, the pair $(A, Q)$ is a relative $d$-Auslander--Gorenstein pair if and only if the pair $(A^{op}, DQ^\natural)$ is.
\end{Remark}

More generally, we have the following:

\begin{Remark}
	Let $(A, Q)$ be a relative $n$-Auslander--Gorenstein pair with $Q\in Q^\perp$, $\pdim_{A}Q \leq d$ and $\injdim_A Q\leq n-d$. Then, $\injdim_AA\leq n$, $\injdim A_A\leq n$ and $DQ\ldom_{A^{op}} A=Q\ldom_A A\geq n$. That is, $(A^{op}, DQ)$ is a relative $n$-Auslander--Gorenstein pair with $\pdim_A DQ\leq n-d$ and $\injdim_A DQ\leq d$.

In particular, $(A,Q)$ is a relative $n$-Auslander pair if and only if $(A^{op}, DQ)$ is a relative $n$-Auslander pair.
\end{Remark}

\section{Complete cotorsion pairs and global dimension}\label{sec5}

Let $A$ be a finite-dimensional algebra. Let $\mathcal{X}$ be a subcategory of $A\m$. We denote by $\widehat{\mathcal{X}}$ the subcategory of $A\m$ consisting of the $A$-modules $X$ for which there is an exact sequence $0\to X_n\to \cdots \to X_1\to X_0\to X\to 0$ for some $n\geq 0$ and with the $X_i$ in $\mathcal{X}$. We also denote by $\widecheck{\mathcal{X}}$ the subcategory of $A\m$ whose modules are the $X$ in $A\m$ such that there is an exact sequence $0\to X\to X_0\to X_1\to \cdots \to X_m\to 0$ for some $m\geq 0$ with all $X_i$ in $\mathcal{X}$.

It is known by Auslander-Reiten \citep{AuslanderReiten} that there is a bijective correspondence between basic tilting modules $T$ and complete cotorsion pairs $(\mathcal{X}, \mathcal{Y})$ where $\mathcal{Y}$ is coresolving and $\widecheck{\mathcal{Y}}=A\m$. Moreover, there is a bijective correspondence between basic cotilting modules $T$ and complete cotorsion pairs $(\mathcal{X}, \mathcal{Y})$ where $\mathcal{X}$ is resolving and $\widehat{\mathcal{X}}=A\m$. Furthermore, the tilting modules over an algebra of finite global dimension coincide with the cotilting modules.

In this section, we improve the aforementioned result of Auslander-Reiten by proving that over an algebra $A$ of finite global dimension the associated complete cotorsion pairs of the tilting and cotilting module, respectively, coincide and basically this is characterized by $A$ having finite global dimension.

\begin{Theorem}\label{thm4dot1}
	Let $A$ be a finite-dimensional algebra and $T$ a basic tilting $A$-module. The following assertions are equivalent:
	\begin{enumerate}[(i)]
		\item $A$ has finite global dimension;
		\item $T$ is a cotilting $A$-module  and $\widecheck{\widehat{\add_A T}}=A\m$;
		\item $T$ is a cotilting $A$-module and the complete cotorsion pairs $(\widecheck{\add_A T}, T^\perp)$ and $(\lerp T, \widehat{\add_A T})$ coincide;
		\item $T$ is a cotilting $A$-module and $T^\perp=\widehat{\add_A T}$;
		\item $T$ is a cotilting $A$-module  and $\widecheck{\add_A T}=\lerp T$;
		\item $T$ is a cotilting $A$-module and $A\m=\widehat{\widecheck{\add_A T}}$.
	\end{enumerate}
\end{Theorem}
\begin{proof}
The implications (iii)$\implies$(iv), (iii)$\implies$(v) are clear. Assume that (iv) holds. Then $\widehat{\add_A T}$ is a coresolving subcategory of $A\m$ and by \citep[Corollary 2.3(a)]{R} $\widecheck{T^\perp}=A\m$, so $(ii)$ is immediate. Analogously, it follows $(v)\implies (vi)$.

We show (vi)$\implies$(i). Let $X\in \widecheck{\add_A T}$. Then, there exists an exact sequence $0\rightarrow X\rightarrow T_0\rightarrow \cdots\rightarrow T_r\rightarrow 0$ with $T_i\in \add_A T$. Denote by $X_i$ the kernel of $T_i\rightarrow T_{i+1}$, $i=0, \ldots, r-1$. Since $T$ has finite projective dimension, so also has $X_{r-1}$. By induction on the exact sequences $0\rightarrow X_i\rightarrow T_i\rightarrow X_{i+1}\rightarrow 0$, all $X_i$, $i=0, \ldots, r-2$, have finite projective dimension. In particular, $X$ has finite projective dimension over $A$. Therefore, any module $Y\in A\m=\widehat{\widecheck{\add_A T}}$ has a finite resolution by modules having finite projective dimension over $A$. It follows that any $A$-module has finite projective dimension.

Assume that (i) holds. It follows from $DA\in T^\perp$, $DA$ having finite projective dimension and $T$ having finite injective dimension that $T$ is a cotilting $A$-module. In fact, the former argument implies that $DA\in \widehat{\add_A T}$. 
	Since $T$ is a cotilting $A$-module by \citep[Corollary 2.3(b)]{R}, $\widehat{\add_A T}$ is a coresolving subcategory of $A\m$. In particular, $\widehat{\add_A T}$ contains all injective $A$-modules. Since $\gldim A$ is finite, it follows that $\widecheck{\widehat{\add_A T}}=A\m$. So, (ii) holds.
	
	Finally, assume that (ii) holds. We aim to show that (iii) also holds. Since $T$ is cotilting, we obtain by the cotilting correspondence (see \citep[Corollary 2.3(b)]{R}) that $(\lerp T, \widehat{\add_A T})$ is a complete cotorsion pair. In particular, $\widehat{\add_A T}$ is a coresolving subcategory of $A\m$ satisfying $\widecheck{\widehat{\add_A T}}=A\m$. Hence by \citep[Corollary 2.3(a)]{R}, there exists a unique basic tilting $A$-module $U$ satisfying \begin{align}
		(\widecheck{\add_A U}, U^\perp)=(\lerp T, \widehat{\add_A T}).
	\end{align}
In particular, $U\in \widecheck{\add_A U}=\lerp T$, that is, $\Ext_A^{i>0}(U, T)=0$. Furthermore, $U\in U^\perp=\widehat{\add_A T}$. So, there exists an $A$-exact sequence $0\rightarrow T_r\rightarrow \cdots \rightarrow T_0\rightarrow U\rightarrow 0$, where $T_i\in \add_AT$. By applying $\Hom_A(T, -)$ to this exact sequence we deduce that $\Ext_A^{i>0}(T, U)=0$. Thus, $\Ext_A^{i>0}(U\oplus T, U\oplus T)=0$, $U\oplus T$ has finite projective dimension and clearly $A\in \widecheck{\add_A T}\subset \widecheck{\add(T\oplus U)}$. This means that $T\oplus U$ is a tilting $A$-module, and so it has the same number of non-isomorphic indecomposable summands as $T$, and as $U$. Hence, $T\simeq U$.
\end{proof}

\begin{Example}
        Let $A$ be the following bound quiver algebra
        \begin{equation}
                \begin{tikzcd}
                        1 \arrow[r, "\gamma"] \arrow[out=150,in=210,loop,"\alpha", swap] & 2 \arrow[out=-30,in=30,loop,swap,"\beta"]
                \end{tikzcd}, \quad
                \begin{aligned}
\alpha^2=\beta^2=0, \ \gamma\alpha=\beta\gamma.
                \end{aligned}
        \end{equation}
The projective cover of the simple module $1$ is also an injective module with socle $2$, whereas the injective hull of $1$, $I(1)$, fits into an exact sequence $0\rightarrow 1\rightarrow I(1)\rightarrow 1\rightarrow 0$. Furthermore, $1$ has infinite injective dimension, and consequently, $A$ has infinite global dimension. Moreover, $A$ is an Iwanaga-Gorenstein algebra with Gorenstein dimension one. Fix $T$ to be the direct sum of all non-isomorphic indecomposable injective $A$-modules. So, $T$ is a tilting-cotilting module with injective dimension zero and projective dimension one.
Observe that $\widehat{\add_A T}$ is the subcategory of $A\m$ of all modules of finite injective dimension, $\widecheck{\add_A T}$ is the subcategory of $A\m$ of all injective modules, ${}^\perp T=A\m$ and $T^\perp$ is the subcategory of $A\m$ formed of all Gorenstein injective $A$-modules. It follows that $\widehat{\widecheck{\add_A T}}=\widecheck{\widehat{\add_A T}}$ is exactly the subcategory of all modules of finite injective dimension which is strictly contained in $A\m$. We can also see that there are Gorenstein injective $A$-modules which are not $A$-injective. For example, $1$ is a Gorenstein injective $A$-module. In fact, applying $\Hom_A(-, I(1))$ to the minimal projective resolution of $I(1)$, we obtain that $\Ext_A^1(T, I(1))=0$.
\end{Example}

\section{Relative Gorenstein-Auslander Pairs}
\label{relativeGorAuspairs}

In this section, we will prove the main result of this paper. In  this direction, we start with the following result.

\begin{Lemma}
\label{thmQdomdimTd}
Let $(A, Q)$ be a $n$-relative Auslander--Gorenstein pair with $\pdim_A Q\leq d$ and ${\injdim_A Q\leq n-d}$ for some non-negative number $d$. Then there exists a basic $d$-tilting module $T$ satisfying $${Q\ldom_A T\geq n-d} \quad \text{and} \quad Q\lcodom_A T\geq 1.$$ 
\begin{proof}
If $d=0$, then the claim holds for $T:=A$. If $d=n$, then the claim follows for $T:=DA$. Suppose now that $0<d<n$. By assumption, there exists an exact sequence 
\begin{equation}
\label{exactseqQ2d}
0\rightarrow A\rightarrow Q_1\rightarrow \cdots \rightarrow Q_{n} \rightarrow Z\rightarrow 0
\end{equation}
which remains exact under $\Hom_A(-,Q)$. Denote by $X$ the image of $Q_d \to Q_{d+1}$. By dimension shift we have $\Ext^i_A(X, Q)\cong \Ext^{i+n-d}_A(Z, Q)=0$ for all $i\geq 1$ since $\injdim_AQ\leq n-d$. Hence $X\in {^\perp}Q$. From the exact sequence $0\to X\to Q_{d+1}\to \cdots \to Q_{n}$ it follows that $Q\ldom_AX\geq n-d$. From the exact sequence $0\rightarrow A\rightarrow Q_1\rightarrow \cdots \rightarrow Q_{d} \rightarrow X\rightarrow 0$, we obtain that $\Ext^i_{A}(Q, X)\cong 
\Ext^{i+d}_A(Q, A)=0$ for all $i>0$ since  $\pdim_A Q\leq d$. This implies that the exact sequence $0\to X\to Q_{d+1}\to \cdots \to Q_{n}\to Z\to 0$ remains exact under $\Hom_A(Q,-)$. So $Q\lcodom_AZ\geq n-d$ and therefore by the dual of Lemma~\ref{lemmaextendingdom} we get that $Q\lcodom_AZ\geq n-d+1$. This shows that $Q\lcodom_A{X}\geq 1$. By the sequence $(\ref{exactseqQ2d})$ and using that $\pdim_AQ\leq d<\infty$ we also get that $\pdim_AZ<\infty$. Since the algebra $A$ is $n$-Iwanaga-Gorenstein, it follows that  $\findim{A}\leq n$ (see for instance \cite[Proposition~6.10]{AuslanderReiten}) and therefore $\pdim_AZ\leq n$. By dimension shift we have $\pdim_AX\leq d$. Set $T:=Q\oplus X$. We have shown that $\pdim_AT\leq d$. Moreover, we have observed that $X\in {^\perp}Q$ and $X\in Q^{\perp}$. As $\Ext^i_{A}(X,X)\cong \Ext^{i+d}_A(X,A)=0$ for all $i>0$ since $\pdim_AX\leq d$ we infer that $\Ext^i_{A}(T, T)=0$ for all $i>0$. We conclude that $T$ is a $d$-tilting module such that $Q\ldom_A T\geq n-d$ and $Q\lcodom_A T\geq 1$.
\end{proof}
\end{Lemma}

We will now restrict to the case of finite-dimensional algebras with finite global dimension.

\begin{Prop}
\label{prop QdomdimA}
Let $A$ be a finite-dimensional algebra with finite global dimension. Let $Q\in A\m$ satisfying $\Ext_A^{i>0}(Q,Q)=0$, $\pdim_A Q\leq d$ and $\injdim_A Q\leq n-d$ for some number $d\in \{0, \ldots, n\}$. Assume that there exists a $d$-tilting and $(n-d)$-cotilting module $_AT$ satisfying 
\[
Q\ldom_A T\geq n-d \ \ \text{and} \ \ Q\lcodom_A T \geq d.
\]
Then we have: $Q\ldom_AA \geq n \geq \gldim A$.
\end{Prop}
\begin{proof}
Assume that $d=0$. Since $T$ is a $0$-tilting module it is clear that $\add T=\add A$ and thus the result is clear for $d=0$. For $d=n$, $T$ is a $0$-cotilting module and in such a case $\add T=\add DA$. Hence, the result is also clear for $d=n$. Assume that $0<d<n$.
By Remark~\ref{remconditiondualadd} (iii), the module $Q$ lies in $\add_AT$. Consider the cotilting resolution
\[
0\to T_{n-d}\to \cdots \to T_0 \to DA\to 0. 
\]
It remains exact under $\Hom_A(Q, -)$ and $T_i\in \add_A T$. Then it follows that $Q\lcodom_A DA\geq d$ and therefore $Q\ldom_AA\geq d$. So, there exists an exact sequence 
\begin{align}
	0\to A\to Q_1\to \cdots \to Q_d\to X\to 0 \label{eq23}
\end{align}
which remains exact under $\Hom_A(-,Q)$ and $Q_i\in \add_A Q$. Since $\pdim_AQ<\infty$ we get that $\pdim_AX<\infty$. Note that $\pdim_AX\geq d$ since otherwise the exact sequence (\ref{eq23}) splits. Also the exactness of (\ref{eq23}) gives that $\Ext_A^i(X, Q)=0$ for $i=1, \ldots, d$ and $\Ext_A^{i+d}(X, Q)\simeq \Ext_A^i(A, Q)=0$. Thus, $X\in \lerp Q$.

 Since $Q\lcodom_A T \geq d$, there exists an exact sequence 
\begin{align}
0\to Y\to Q_{d}'\to \cdots \to Q_1' \to T\to 0  \label{eq_24}
\end{align}
which remains exact under $\Hom_A(Q,-)$ and $Q_i'\in \add_A Q$. 
The exactness of (\ref{eq_24}) together with $\pdim_AQ\leq d$ gives that $Y\in Q^{\perp}$.
Using the exact sequence (\ref{eq23}), it follows that $\pdim_{Q^{\perp}}X\leq d + \pdim_{Q^{\perp}}A=d$ and therefore $\Ext_A^{i+d}(X,Y)=0$ for every $i>0$. Observe that Lemma~\ref{lemma1dot1} gives that $\Ext^i_A(X, T)\cong \Ext^{i+d}_A(X, Y)$ for all $i>0$. Thus, $X\in \lerp T$.

Since $\gldim A<\infty$ the module $X$ lies in $\widecheck{\add_A T}$ by Theorem~\ref{thm4dot1} and therefore we can extend the exact sequence $(\ref{eq23})$ using summands of the tilting module $T$. But if we extend it one more step, then it splits because $\pdim_AT\leq d$. We infer that $X\in \add_AT$.

Now, the exact sequence (\ref{eq23}) gives that $Q\ldom_A A=Q\ldom_A X+d\geq Q\ldom_A T+ d \geq n-d+d=n$. For each ${N\in A\m}$, applying $\Hom_A(N, -)$ to (\ref{eq23}) yields that $\Ext_A^{d+\injdim_A T+1}(N, A)\simeq \Ext_A^{\injdim_A T +1}(N, X)=0$ since $\injdim_A X\leq \injdim_A T$. Thus, $\gldim A =\injdim_A A\leq \injdim_A T+d\leq n$.
\end{proof}

We will now prove the second half of the main result.

\begin{Prop}
\label{thmcotiltilbydimd}
Let $Q\in A\m$ satisfying $\Ext_A^{i>0}(Q,Q)=0$, $\pdim_A Q\leq d$ and $\injdim_A Q\leq n-d$ for a number $d\in \{0, \ldots, n\}$. Assume that $\gldim{A}\leq n\leq Q\ldom{A}$. Then there exists a unique basic $d$-tilting and $(n-d)$-cotilting module $T$ satisfying $Q\ldom_A T\geq n-d$ and $Q\lcodom_A T\geq d$.
\end{Prop}
\begin{proof}
For $d=0$ we fix $T$ the basic module satisfying $\add T=\add A$ while for $d=n$ we fix $T$ a basic module satisfying $\add T=\add DA$. The uniqueness for such cases is clear for example by Lemma \ref{lemma1dot4}. Let $d$ be a natural number smaller than $n$. 
By Lemma~\ref{thmQdomdimTd}, there exists a $d$-tilting module $T$ with $Q\ldom_A T\geq n-d$. Recall that ${T=Q\oplus X}$, where $X$ fits in the exact sequence 
\begin{equation}
	\label{exactseqforA}
	0\to A\to Q_1\to \cdots \to Q_d\to X\to 0
\end{equation} which remains exact under $\Hom_A(-, Q)$ and $Q_i\in \add Q$. Note also that
\[
\gldim{A}^{op}\leq n\leq Q\ldom{A}=Q\lcodom{DA}=DQ\ldom{A^{op}}
\]
and therefore Theorem~\ref{thmQdomdimTd} gives a $(n-d)$-tilting module $T_{op}$ with $DQ\ldom_{A^{op}}T_{op}\geq n-(n-d)=d$. Set $C:=DT_{op}$ and note that $C=Y\oplus Q$ where $Y$ appears in the following exact sequence:
\begin{equation}
\label{exactseqforDA}
0\rightarrow Y\rightarrow Q'_{n-d}\rightarrow \cdots \rightarrow Q'_1 \rightarrow DA\rightarrow 0
\end{equation}
with $Q_i\in \add Q$ since $Q\lcodom{DA}\geq n$. So, the module $C$ is $(n-d)$-cotilting with $Q\lcodom_AC\geq d$. We shall proceed to show that $\add_A C=\add_A T$. By construction we have $Q\in\add_A{C}$ and $Q\in \add_{A}T$. Using the exact sequence (\ref{exactseqforA})
together with $\injdim_AA\leq n$ and $\injdim_AQ\leq n-d$, it follows that $\injdim_AX\leq n-d$ and therefore $\injdim_AT\leq n-d$. Moreover, using the exact sequence $(\ref{exactseqforDA})$ together with $\pdim_ADA\leq n$ and $\pdim_AQ_i'\leq d$, we obtain that $\pdim_AY\leq d$ and therefore $\pdim_AC\leq d$. Hence, we infer that both $C$ and $T$ are $d$-tilting and $(n-d)$-cotilting modules and both of them belong to ${{^\perp}Q}$.

Since $C$ lies in ${{^\perp}Q}$ and $\pdim_AC\leq d$ applying Lemma~\ref{lemma1dot1} on the exact sequence (\ref{exactseqforA}), we obtain that $\Ext^i_A(C, X)\cong \Ext_A^{i+d}(C, A)=0$ for all $i>0$.

 Thus, since $T=Q\oplus X$ and $\Ext_A^{i>0}(C,Q)=0$ we have $\Ext_A^{i>0}(C,T)=0$. From Theorem~\ref{thm4dot1} we have that $T$ belongs to $C^{\perp}=\widehat{\add_A C}$. This means that there exists an exact sequence $0\to C_s\to \cdots \to C_0\to T\to 0$ for some $s\geq 0$ and $C_i\in \add{C}$. This exact sequence remains exact under $\Hom_A(Q,-)$ because $C_i, T\in Q^\perp$. Hence, since $Q\lcodom_AC_i\geq d$ for all $0\leq i\leq s$ we conclude that $Q\lcodom_AT\geq d$. Analogously, $DQ\lcodom_{A^{op}} DC\geq n-d$. The uniqueness now follows by Lemma \ref{lemma1dot4} and in particular $\add_AC=\add_A T$.
\end{proof}

Actually, the statement of Proposition \ref{thmcotiltilbydimd} remains valid for general relative Auslander--Gorenstein pairs. 

\begin{Prop}
\label{prop: improvement}
Let $Q\in A\m$ satisfying $\Ext_A^{i>0}(Q,Q)=0$, $\pdim_A Q\leq d$ and $\injdim_A Q\leq n-d$ for a non-negative number $d\in \{0, \ldots, n\}$. Assume that  $(A, Q)$ is a relative $n$-Auslander-Gorenstein pair. Then there exists a unique basic $d$-tilting and $(n-d)$-cotilting module $T$ satisfying $Q\ldom_A T\geq n-d$ and $Q\lcodom_A T\geq d$.
\end{Prop}
\begin{proof} 
Like in Proposition \ref{thmcotiltilbydimd}, the result is clear for $d\in \{0, n\}$. Suppose that $1\leq d\leq n-1$.
By Lemma~\ref{thmQdomdimTd}, there exists a basic $d$-tilting module $T$ satisfying $Q\ldom_A T\geq n-d$ and there exists an exact sequence
\[
0\to A\to Q_0\to \cdots \to Q_{d-1}\to X\to 0
\] 
with $\add{T}=\add{Q\oplus X}$. Since $\injdim_A Q\leq n-d$ and $\injdim_A A\leq n$ we obtain that $\injdim_A X\leq n-d$. So $T$ is also a $(n-d)$-cotilting module. Consider the minimal injective resolution of $T$:
\[
0\to T\to I_0\to I_1 \cdots \to I_{t}\to 0
\] 
with $t\leq n-d$. Since $Q$ lies in $\add{T}$, the above sequence remains exact under $\Hom_A(Q,-)$. Using the fact that $Q\lcodom_A I_i\geq n$ we get from \cite[Lemma~3.1.7]{Cr2} that 
\[
Q\lcodom_A \mathsf{Ker}(I_{t-1}\to I_t)\geq n-1.
\]
We infer that $Q\lcodom_A T\geq n-t\geq n-(n-d) = d$.
\end{proof}

The two approaches of Proposition~\ref{thmcotiltilbydimd} and Proposition~\ref{prop: improvement} yield the following identification:

\begin{Cor}
\label{corollary of the two props}
Let $Q\in A\m$ satisfying $\Ext_A^{i>0}(Q,Q)=0$, $\pdim_A Q\leq d$ and $\injdim_A Q\leq n-d$ for some natural number $d$ smaller than $n$. Assume that  $(A, Q)$ is a relative $n$-Auslander-Gorenstein pair. Then $\add{Q\oplus X} = \add{Q\oplus Y}$ where $X$ fits into an exact sequence 
\[
0\to A\to Q_1\to \cdots \to Q_{d}\to X\to 0
\] 
which remains exact under $\Hom_A(-,Q)$ and $Y$ fits into an exact sequence 
\[
0\to Y \to Q_{n-d}'\to \cdots \to Q_{1}'\to DA\to 0
\]
which remains exact under $\Hom_A(Q,-)$ with $Q_i, Q_i'\in \add_A Q$ for all $i$.
\end{Cor}
\begin{proof}
By Proposition~\ref{prop: improvement}, the module $Q\oplus X$ is $d$-tilting and $(n-d)$-cotilting with $Q\ldom_A T\geq n-d$ and $Q\lcodom_A T\geq d$. Dualising and since $(A^{\op}, DQ)$ is a relative $n$-Auslander-Gorenstein pair we get that the module $DQ\oplus DY$ is $(n-d)$-tilting and $d$-cotilting with $DQ\ldom_{A^{\op}} DQ\oplus DY\geq n-(n-d)=d$ and $DQ\lcodom_{A^{\op}} DQ\oplus DY\geq n-d$. This implies that $Q\oplus Y$ is a $(n-d)$-cotilting and $d$-tilting module satisfying $Q\lcodom_{A} Q\oplus Y\geq d$ and $Q\ldom_{A} Q\oplus Y\geq n-d$. By Lemma~\ref{lemma1dot4}, we infer that $\add{Q\oplus Y}=\add{Q\oplus X}$.
\end{proof}

Now combining the previous results we obtain a characterization of relative $n$-Auslander pairs in terms of the existence of a $d$-tilting $(n-d)$-cotilting module with the property that the sum of relative dominant dimension with respect to $Q$ with the relative codominant dimension with respect to $Q$ is greater or equal to $n$. 

\begin{Theorem}\label{mainthm}
Let $Q\in A\m$ with $Q\in Q^{\perp}$, $\pdim_{A}Q\leq d$, $\injdim_AQ\leq n-d$ for some non-negative numbers $n$ and $d$ with $n\geq d$. Then there exists a unique basic $d$-tilting $(n-d)$-cotilting module $T$ satisfying the following conditions
\begin{enumerate}[(i)]
\item $Q\ldom_AT\geq n-d$, 

\item $Q\lcodom_AT\geq d$, and

\item  $T^\perp=\widehat{\add_A T}$
\end{enumerate}
if and only if 	$(A, Q)$ is a relative $n$-Auslander pair.
\end{Theorem}
\begin{proof}
Assume that $(A, Q)$ is a relative $n$-Auslander pair. The existence of the tilting-cotilting module in the desired conditions follows from Proposition~\ref{thmcotiltilbydimd} and Theorem~\ref{thm4dot1}. Conversely, if there exists a basic $d$-tilting-cotilting module satisfying $(i), (ii), (iii)$, then $A$ has finite global dimension by Theorem \ref{thm4dot1}. By Proposition \ref{prop QdomdimA}, the result follows.
\end{proof}

It is of particular interest, the case where $n$ is exactly $2d$ characterising the even relative Auslander pairs. This result also generalises \citep[Lemma 1.1]{zbMATH06685118} by fixing $Q$ to be projective-injective and $d=1$.

 Theorem \ref{mainthm} also generalises \citep[Theorem 3.2]{zbMATH07441895} in the situation that $A$ has finite global dimension. Another consequence of Theorem \ref{mainthm} is that we can associate more than one tilting-cotilting to a relative Auslander-Gorenstein pair if $\pdim_A Q+\injdim_AQ$ is smaller than the global dimension of $A$, but each choice of $d$ so that $\pdim_A Q\leq d$ and $\injdim_A Q\leq n-d$ determines uniquely the tilting-cotilting having projective dimension at most $d$ and injective dimension at most $n-d$ associated with the relative Auslander-Gorenstein pair.
	
We leave as open question whether Proposition \ref{prop QdomdimA} remains valid  for Iwanaga-Gorenstein algebras.

\section{Hemmer-Nakano type results}\label{sec7}
In this section, we prove a Hemmer-Nakano type result for the Schur functor from the module category of a finite-dimensional algebra $A$ with a module $Q$ satisfying the properties of Theorem~\ref{thmcotiltilbydimd} and the module category of the endomorphism algebra of the unique basic tilting-cotilting module (as constructed in Theorem~\ref{thmcotiltilbydimd}).

Define $Q\lcodom_A \widehat{\add T}:=\inf\{Q\lcodom_A M \ | \ M\in \widehat{\add T}\}$.  The following is the non-quasi-hereditary version of \citep[Theorem 5.3.1(a)]{Cr2}.
\begin{Lemma}\label{lemma5dot1}	Let $Q\in A\m$ satisfying $\Ext_A^{i>0}(Q,Q)=0$.
	Assume that $T$ is a tilting module with $T^\perp=\widehat{\add T}$. Then, \begin{align}
		Q\lcodom_A T=Q\lcodom_A \, \widehat{\add T}.
	\end{align}
\end{Lemma}
\begin{proof}
	Since $T\in \widehat{\add T}$, $Q\lcodom_A T\geq Q\lcodom_A \widehat{\add T}$. If $Q\lcodom_A \widehat{\add T}$ is infinite, then there is nothing more to show. Assume that $Q\lcodom_A T=n$. Let $M\in \widehat{\add T}$. Then there exists an exact sequence $0\rightarrow T_r\rightarrow \cdots\rightarrow T_0\rightarrow X\rightarrow 0$, where $T_i\in \add T$. In particular, $Q\lcodom_A T_i\geq n$. Applying Lemma 3.1.7 on the short exact sequences of this resolution we obtain, by induction, that $Q\lcodom_A X\geq n$.
\end{proof}

We need the following result which is a generalisation of the classical Brenner-Butler theorem.

\begin{Lemma}
	Assume that $T$ is a tilting-cotilting module with $T^\perp=\widehat{\add T}$. Denote by $E$ the endomorphism algebra $\End_A(T)^{op}$. Then, the following assertions hold.
	\begin{enumerate}
		\item The restriction of the functor $\Hom_A(T, -)\colon A\m\rightarrow E\m$ to $\widehat{\add T}$ is fully faithful;
		\item The functor $\Hom_A(T, -)\colon A\m\rightarrow E\m$  restricts to an exact equivalence between $T^\perp$ and $\{X\in E\m \ | \ \Tor_{i>0}^E(T, X)=0\}$.
	\end{enumerate}
\end{Lemma}
\begin{proof}
	Since $\widehat{\add T}\subset T^\perp$, $\Hom_A(T, -)$ is exact on $\widehat{\add T}$. Observe that for any $X\in A\m$, \begin{align}
		\Hom_A(T, X)\simeq \Hom_E(E, \Hom_A(T, X))\simeq \Hom_E(\Hom_A(T, T), \Hom_A(T, X)).
	\end{align}Pick $X\in \widehat{\add_A T}.$ Then, there exists an exact sequence $T_1\rightarrow T_0\rightarrow X\rightarrow 0$, with $T_0, T_1\in \add_A T$. Applying $\Hom_A(-, Y)$ and $\Hom_E(\Hom_A(T, -), \Hom_A(T, X))$ to this exact sequence we obtain by diagram chasing that $\Hom_A(T, -)$ is fully faithful on $\widehat{\add T}$ (see also \citep[Lemma 2.2.4]{Cr1}). In particular, $\Hom_E(\Hom_A(T, Q), E)\simeq \Hom_A(Q, T)$. Now, for any $X\in \widehat{\add_A T}$, the image of $0\rightarrow T_r\rightarrow \cdots\rightarrow T_0\rightarrow X\rightarrow 0$ under $\Hom_A(T, -)$ is a projective $E$-resolution of $\Hom_A(T, X)$. Since $\Hom_A(T, -)$ is fully faithful on $\widehat{\add_A T}$, the image of the previous projective $E$-resolution under $T\otimes_E -$ is equivalent to the exact sequence $0\rightarrow T_r\rightarrow \cdots\rightarrow T_0\rightarrow X\rightarrow 0$. Therefore, $\Tor_{i>0}^E(T, \Hom_A(T, X))=0$.

Let $Y\in E\m$ such that $\Tor_{i>0}^E(T, Y)=0$. By Theorem \ref{thm4dot1}, $A$ has finite global dimension, and since $E$ is derived equivalent to $A$, $E$ also has finite global dimension. So, applying $T\otimes_E -$ to a finite projective resolution of $Y$, and since $T\otimes_E -$ sends projectives to modules in $\add_A T$ we obtain that by $\Tor_{i>0}^E(T, Y)=0$ that $T\otimes_E Y\in \widehat{\add_A T}=T^\perp$. Since $\Hom_A(T, T\otimes_E P)\simeq P$ as $E$-modules, by applying $\Hom_A(T, T\otimes_E -)$ to a finite projective resolution of $Y$ together with diagram chasing we can deduce that $\Hom_A(T, T\otimes_E Y)\simeq Y$ as $E$-modules. 
\end{proof}

We now recall the notion of a cover with respect to a resolving subcategory,
introduced by  the first-named author \cite[Section~3]{Cr1}.

\begin{Def} 
Let $A$ be a finite-dimensional algebra and let $\mathcal{A}$ be a resolving subcategory of $A\m$. Let $P$ be a finitely generated projective $A$-module and write $B=\End_A(P)^{\op}$ for the endomorphism algebra. Let $i\geq 0$ be a positive integer. The pair $(A, P)$ is called an  {\bf $i\hy\mathcal{A}$ cover} of $B$ if the Schur functor $F=\Hom_A(P,-)\colon A\m\to B\m$ induces isomorphisms
\[
\Ext_A^j(M, N) \to \Ext_A^j(FM, FN)
\]
for all $M$ and $N$ in $\mathcal{A}$ and for all $0\leq j\leq i$.
\end{Def}

For the next result, it is useful to recall the following homological characterisation of relative codominant dimension.

\begin{Theorem}
	Let $Q\in A\m$ and denote by $B$ the endomorphism algebra $\End_A(Q)^{\op}$. For any $M\in A\m$ and natural number $n$, $Q\lcodom_A M\geq n\geq 2$ if and only if $Q\otimes_B \Hom_A(Q, M)\rightarrow M$, $q\otimes f\mapsto f(q)$, is an isomorphism of left $A$-modules and $\Tor_i^B(Q, \Hom_A(Q, M))=0$ for $1\leq i\leq n-2$.
\end{Theorem}
\begin{proof}
	See \citep[Theorem 3.1.4.]{Cr2}.
\end{proof}

Observe that Theorem over finite-dimensional algebras could be reformulated using extension groups. Indeed, according to the notation of Theorem, over finite-dimensional algebras there exists the following identity $D\Tor_{i>0}^B(Q, \Hom_A(Q, M))\simeq \Ext_B^{i>0}(\Hom_A(Q, M), \Hom_A(Q, DA))$ (see for example \citep[equation (200)]{Cr1}).

\begin{Theorem}\label{thm7dot4}
	Let $Q\in A\m$ satisfying $Q\in Q^{\perp}$, $\pdim_A Q\leq d$ and $\injdim_A Q\leq d$.	Assume that $A$ satisfies the following
	\begin{align}
		\gldim A\leq 2d\leq Q\ldom_A A, \label{eq24}
	\end{align} for some $d>1$. 
Let $T$ be the basic $d$-tilting-cotilting module constructed in Proposition~\ref{thmcotiltilbydimd}. Denote by $E$ the endomorphism algebra $\End_A(T)^{\op}$. Then, the following assertions hold.
\begin{enumerate}[(i)]
	\item The restriction of the functor $F_Q:=\Hom_A(Q, -)\colon A\m\rightarrow \End_A(Q)^{\op}\m$ to $\widehat{\add_A T}$ is fully faithful;
	\item The functor $F_Q$ induces canonical isomorphisms for any $M, N\in T^\perp\colon$
	\begin{equation}
		\Ext_A^i(M, N)\simeq \Ext_{\End_A(Q)^{\op}}^i(F_QM, F_QN), \quad 0\leq i\leq d-2;
	\end{equation}
\item The pair $(E, \Hom_A(T, Q))$ is a $(d-2) \hy\mathcal{A}$ cover of $\End_A(Q)^{\op}$ where 
\[
\mathcal{A}=\{Z\in E\m\colon \Tor_{i>0}^E(T, Z)=0\}.
\]
\end{enumerate}
\end{Theorem}
\begin{proof}
The statement (i) follows from Lemma \ref{lemma5dot1} and Lemma 5.1.1 of \citep{Cr2}.
So, in particular, \begin{align}
	\End_B(\Hom_A(Q, T))^{\op}\simeq \End_A(T)^{\op}=E \\
	\End_E(\Hom_A(T, Q))^{\op}\simeq \End_A(Q)^{\op},
\end{align}where the latter follows from projectivization. This shows that $(E, \Hom_A(T, Q))$ is a cover of $\End_A(Q)^{\op}$. The statement (ii) follows from Lemma 5.2.1 of \cite{Cr2}, Lemma \ref{lemma5dot1} and Theorem 3.1.1 of \cite{Cr2}. 

Denote by $B$ the endomorphism algebra $\End_A(Q)^{\op}$. Observe that for any $M\in T^\perp$,
\begin{align}
	\Ext_B^i(\Hom_A(Q, T), \Hom_E(\Hom_A(T, Q), \Hom_A(T, M)))\simeq 
	\Ext_B^i(\Hom_A(Q, T), \Hom_A(Q, M))\\=\Ext_B^i(F_QT, F_QM)\simeq \Ext_A^i(T, M),
\end{align}for all $0\leq i\leq d-2$.

Combining this fact with (i), (iii) follows by \citep[Proposition 3.0.3, Proposition 3.0.4]{Cr1}.
\end{proof}

\section{Examples}
\label{Examples}

We finish the paper with a collection of relative Auslander pairs.

\begin{Example}
	Any pair $(A, M)$ formed by a semi-simple algebra and a faithful $A$-module is a relative $d$-Auslander pair for every natural number $d$.
\end{Example}

\begin{Example}
	Let $A$ be a finite-dimensional algebra with global dimension two and fix $M\in A\m$ satisfying $M\in M^\perp$. Then, it follows from \cite[Lemma 5.1.2, Corollary 3.1.5, Theorem 3.1.1]{Cr2} that the pair $(A, M)$ is a relative $2$-Auslander pair if and only if $M$ has a double centralizer property. Furthermore, every module without self-extensions  over a finite-dimensional algebra of global dimension at most two affording a double centralizer property can be completed to a full tilting module.
\end{Example}

\begin{Example}
	Let $A$ be the following bound quiver algebra
	\begin{equation}
		\begin{tikzcd}
			& 1 \arrow[dr, "\alpha"] \arrow[dl, "\beta", swap]& \\
			2 \arrow[dr, "\gamma", swap] & & 3 \arrow[dl, "\nu"] \\
			& 4 &
		\end{tikzcd}, \quad
		\begin{aligned}
			\nu\alpha=\gamma\beta.
		\end{aligned}
	\end{equation}
	So, $A$ is a finite-dimensional algebra of global dimension exactly two.
	
	For each vertex $i\in \{1, 2, 3, 4\}$, denote by $P(i)$, $I(i)$ and $S(i)$ the projective, injective and simple module associated with the vertex $i$, respectively.
	Pick $Q$ to be the injective $A$-module ${I(2)\oplus I(3)\oplus I(4)}$. It has projective dimension one. Observe that $P(1)\simeq I(4)\in \add_A Q$ and there are exact sequences $0\rightarrow P(2)\rightarrow I(4)\rightarrow I(3)\rightarrow 0$, $0\rightarrow P(3)\rightarrow I(4)\rightarrow I(2)\rightarrow 0$, therefore ${Q\ldom_A P(1)\oplus P(2)\oplus P(3)}$ is infinite since $Q$ is injective.
	Furthermore, the minimal injective resolution of $S(4)$ is $0\rightarrow S(4)\rightarrow I(4)\rightarrow I(2)\oplus I(3)\rightarrow I(1)$. Using \citep[Proposition 3.1.13]{Cr2}, this means that $Q\ldom_A P(4)=2$. This shows that $(A, Q)$ is a relative $2$-Auslander pair.
	The tilting-cotilting module associated with this relative Auslander pair is $T_Q:=Q\oplus I(4)/S(4)$. To see this, note that
	the exact sequence $0\rightarrow S(4)\rightarrow I(4)\rightarrow I(4)/S(4)\rightarrow 0$ yields that $Q\ldom_A Q\oplus I(4)/S(4)\geq 1$ since $Q$ is injective and ${A\in \widecheck{\add_A Q\oplus I(4)/S(4)}}$.
	Again since $Q$ is injective and $I(4)$ is projective, $$\Ext_A^1(T_Q, T_Q)\simeq \Ext_A^1(I(2)\oplus I(3)\oplus I(4)/S(4), I(4)/S(4)).$$
	But observe that all the spaces $\Hom_A(P(1), I(4)/S(4)), \Hom_A(P(2), I(4)/S(4))$, $\Hom_A(P(3), I(4)/S(4))$ have vector space dimension one and $$\Hom_A(I(3), I(4)/S(4))=\Hom_A(I(2), I(4)/S(4))=\Hom_A(S(4), I(4)/S(4))=0.$$ Hence,  $\Ext_A^1(I(2)\oplus I(3)\oplus I(4)/S(4), I(4)/S(4))$ must be zero and the exact sequence $0\rightarrow S(4)\rightarrow I(4)\rightarrow I(4)/S(4)\rightarrow 0$ remains exact under $\Hom_A(Q, -)$. So, $T_Q$ is a 1-tilting-cotilting $A$-module satisfying $Q\ldom_AT_Q, Q\lcodom_A T_Q\geq 1$.
\end{Example}

\begin{Example}Consider the relative Auslander pair of the previous example.
	By the symmetry of relative dominant dimension, the pair $(A^{op}, \Hom_k(Q, k))$ is also a relative 2-Auslander pair, but now $\Hom_k(Q, k)$ is projective with injective dimension one. More explicitly, the pair ${(A, P(1)\oplus P(2)\oplus P(3))}$ is a relative $2$-Auslander pair with the associated tilting-cotilting being $P(1)\oplus P(2)\oplus P(3)\oplus \operatorname{rad} P(1)$.
\end{Example}

\begin{Example}
	We exhibit now an example of a relative Auslander pair $(A, Q)$ with $Q$ not being projective nor injective.     Let $A$ be the following bound quiver algebra
	\begin{equation}
		\begin{tikzcd}
			& 2 \arrow[dr, "\gamma"] & & &\\
			1 \arrow[dr, "\alpha", swap] \arrow[ur, "\beta"] & & 4  & 5 \arrow[l, "\theta", swap] & 6 \arrow[l, "\varepsilon", swap]\\
			& 3 \arrow[ur, "\nu", swap]&  & &
		\end{tikzcd}, \quad
		\nu\alpha=\gamma\beta, \ \theta\varepsilon=0.
	\end{equation}
	As in the previous examples, $A$ has global dimension two and we will use the notation $P(i)$, $I(i)$ and $S(i)$ to write the projective, injective and simple module associated with the vertex $i$, respectively.
	Pick $Q$ to be the $A$-module $I(3)\oplus I(2)\oplus S(5)\oplus I(5)\oplus I(4)$. It has projective dimension and injective dimension one and $Q\ldom S(4)\oplus P(5)= 2$ while the other projective indecomposable modules have infinite relative dominant dimension with respect to $Q$.
\end{Example}

In the following example, we illustrate that there exist relative Auslander pairs $(A, Q)$ with $A$ having arbitrary even global dimension.

\begin{Example}
	Let $k$ be a field and let $d$ be a natural number. The tensor power $(k^2)^{\otimes d}$ is a right module over the group algebra of symmetric group, $kS_d$, via place permutation. The centralizer of this action, $\End_{kS_d}((k^2)^{\otimes d})$, is the Schur algebra $S_k(2, d)$. Using \cite{zbMATH02105773}, Theorem B of \cite{CE}  says that $(S_k(2, d), (k^2)^{\otimes d})$ is a relative $\gldim S_k(2, d)$-Auslander pair. The computation of global dimension of Schur algebras $S_k(2, d)$ can be found in \citep[Theorem 3.7]{P}. In particular, if $k$ has characteristic two and $d$ is even, then $(S_k(2, d), (k^2)^{\otimes d})$ is a relative $d$-Auslander pair.
\end{Example}

\section*{Acknowledgments}

The authors would like to thank the hospitality of the University of Stuttgart where part of this work was developed.

\bibliographystyle{alphaurl}
\bibliography{ref}

\end{document}